\documentclass[reqno]{amsart}
\usepackage{amsmath}
\usepackage{amssymb}
\usepackage{amsthm}
\begin{document}
\def\eq#1{{\rm(\ref{#1})}}
\theoremstyle{plain}
\newtheorem*{theo}{Theorem}
\newtheorem*{ack}{Acknowledgements}
\newtheorem*{pro}{Proposition}
\newtheorem*{coro}{Corollary}
\newtheorem{thm}{Theorem}[section]
\newtheorem{lem}[thm]{Lemma}
\newtheorem{prop}[thm]{Proposition}
\newtheorem{cor}[thm]{Corollary}
\theoremstyle{definition}
\newtheorem{dfn}[thm]{Definition}
\newtheorem*{rem}{Remark}
\def\coker{\mathop{\rm coker}}
\def\ind{\mathop{\rm ind}}
\def\Re{\mathop{\rm Re}}
\def\vol{\mathop{\rm vol}}
\def\Im{\mathop{\rm Im}}
\def\im{\mathop{\rm im}}
\def\Hol{{\textstyle\mathop{\rm Hol}}}
\def\C{{\mathbin{\mathbb C}}}
\def\R{{\mathbin{\mathbb R}}}
\def\N{{\mathbin{\mathbb N}}}
\def\Z{{\mathbin{\mathbb Z}}}
\def\O{{\mathbin{\mathbb O}}}
\def\L{{\mathbin{\mathcal L}}}
\def\X{{\mathbin{\mathcal X}}}
\def\al{\alpha}
\def\be{\beta}
\def\ga{\gamma}
\def\de{\delta}
\def\ep{\epsilon}
\def\io{\iota}
\def\ka{\kappa}
\def\la{\lambda}
\def\ze{\zeta}
\def\th{\theta}
\def\vt{\vartheta}
\def\vp{\varphi}
\def\si{\sigma}
\def\up{\upsilon}
\def\om{\omega}
\def\De{\Delta}
\def\Ga{\Gamma}
\def\Th{\Theta}
\def\La{\Lambda}
\def\Om{\Omega}
\def\Up{\Upsilon}
\def\sm{\setminus}
\def\na{\nabla}
\def\pd{\partial}
\def\op{\oplus}
\def\ot{\otimes}
\def\bigop{\bigoplus}
\def\iy{\infty}
\def\ra{\rightarrow}
\def\longra{\longrightarrow}
\def\dashra{\dashrightarrow}
\def\t{\times}
\def\w{\wedge}
\def\bigw{\bigwedge}
\def\d{{\rm d}}
\def\bs{\boldsymbol}
\def\ci{\circ}
\def\ti{\tilde}
\def\ov{\overline}
\def\sv{\star\vp}
\def\stvo{\star_1\vp_1}
\def\stvt{\star_2\vp_2}
\title[CoClosed $G_2$-Structures]{Diffeomorphisms of $7$-Manifolds with CoClosed $G_2$-Structure}

\author[Salur, and Todd]{Sema Salur, and A. J. Todd}

\address {Department of Mathematics, University of Rochester, Rochester, NY, 14627}
\email{salur@math.rochester.edu}

\address {Department of Mathematics, University of California - Riverside, Riverside, CA, 92521}
\email{ajtodd@math.ucr.edu}

\begin{abstract}
We introduce co$G_2$-vector fields, coRochesterian $2$-forms and coRochesterian vector fields on manifolds with a coclosed $G_2$-structure as a continuation of work from \cite{CST1}, and we show that the spaces $\X_{coG_2}$ of co$G_2$-vector fields and $\X_{coRoc}$ of coRochesterian vector fields are Lie subalgebras of the Lie algebra of vector fields with the standard Lie bracket. We also define a bracket operation on the space of coRochesterian $2$-forms $\Om^2_{coRoc}$ associated to the space of coRochesterian vector fields and prove, despite the lack of a Jacobi identity, a relationship between this bracket and so-called \emph{co$G_2$-morphisms}.
\end{abstract}

\date{}
\maketitle
\section*{Introduction}
This paper is a natural continuation of the papers \cite{CST1, CST2} wherein we considered $7$-manifolds with (closed) $G_2$-structures. This is part of a program to better understand the exceptional geometries by viewing them as analogues of symplectic geometry; moreover, this philosophy can be viewed as a continuation of the works of Brown, Fernandez and Gray \cite{Gray, BrGr, Fe1, Fe2, FeIg, FeGr} where symplectic, $G_2$ and $Spin(7)$ geometries all come from the study of cross products on manifolds. In particular, \cite{FeGr} proves that there exist $16$ subclasses of $G_2$-structures including closed $G_2$-structures, coclosed $G_2$-structures as well as $G_2$-structures which are both closed and coclosed. In our previous article, we considered analogues of Hamiltonian and symplectic vector fields on manifolds with closed $G_2$-structure which we called \emph{Rochesterian} and \emph{$G_2$}-vector fields respectively; the current article is concerned with developing these ideas in the case of a coclosed $G_2$-structure. Note that the fundamental $3$-form $\vp$ is coclosed if and only if the associated $4$-form $\star\vp$ is closed. Here $\star\vp$ is the $4$-form which is Hodge-dual to $\vp$ with respect to the Hodge star operator associated to the metric defined by $\vp$. This is in contrast to the symplectic case where the fundamental $2$-form being closed is equivalent to it being coclosed.

The main purpose of this paper is to define \emph{co$G_2$-vector fields}, \emph{coRochesterian $2$-forms} and \emph{coRochesterian vector fields} for coclosed $G_2$-structures and prove the following results:

\begin{theo}
Every coRochesterian vector field on a manifold $M$ with coclosed $G_2$-structure $\vp$ is a co$G_2$-vector field. If every closed form in $\Om^3_7(M)$ is exact, then the spaces $\X_{coRoc}(M)$ and $\X_{coG_2}(M)$ coincide.
\end{theo}

\begin{coro}
If $H^3(M)=\{0\}$, then every co$G_2$-vector field on a manifold with coclosed $G_2$-structure is a coRochesterian vector field.
\end{coro}

We next show that the spaces of co$G_2$- and coRochesterian vector fields admit the structure of Lie algebras with Lie bracket induced from the standard Lie bracket structure on the space of all vector fields and prove the following result on inclusions:

\begin{pro}
For any co$G_2$-vector fields $X_1$, $X_2$, $[X_1,X_2]$ is a coRochesterian vector field with associated coRochesterian $2$-form given by $\sv(X_2,X_1,\cdot,\cdot)$.
\end{pro}

Finally, we equip the space of coRochesterian $2$-forms with a bracket structure analogous to that of the Poisson bracket from symplectic geometry, show that it does \emph{not} satisfy the Jacobi identity, show that there is a linear transformation $\Phi$ of the \emph{vector spaces} of coRochesterian $2$-forms and coRochesterian vector fields and prove the following result regarding these structures:

\begin{theo}
\begin{enumerate}
    \item Given two coRochesterian $2$-forms $\si_1,\si_2\in\Om^2_{Roc}(M)$, $\{\si_1,\si_2\}\in\ker\Phi$ if and only if $\d\si_1$ is constant along the flow lines of $X_{\si_2}$ if and only if $\d\si_2$ is constant along the flow lines of $X_{\si_1}$.
    \item Let $\psi:(M_1,\vp_1)\to(M_2,\vp_2)$ be a diffeomorphism. Then $\psi$ is a co$G_2$-morphism if and only if $\psi^*(\{\si,\tau\})=\{\psi^*\si,\psi^*\tau\}$ for all $\si,\tau\in\Om^2_{coRoc}(M_2)$.
\end{enumerate}
\end{theo}

\section{$G_2$ Geometry}
If we consider coordinates $(x_1,\ldots,x_7)$ on $\R^7$, we can define a $3$-form $\vp_0$ by $$\vp_0=\d x^{123}+\d x^{145}+\d x^{167}+\d x^{246}-\d x^{257}-\d x^{347}-\d x^{356}.$$ From this $3$-form, we get an induced metric and orientation by the formula $$(X\lrcorner\vp_0)\w(Y\lrcorner\vp_0)\w\vp_0=6<X,Y>_{\vp_0}\d vol_{\vp_0}$$ for vector fields $X,Y\in \X(\R^7)$. Then, using this metric, we can define a $2$-fold vector cross product of $X$ and $Y$ as the unique vector field $X\t Y$ satisfying $<X\t Y,Z>_{\vp_0}=\vp_0(X,Y,Z)$ for all $Z\in\X(\R^7)$. This metric then gives the associated Hodge star from which we get the dual of $\vp_0$ given by the $4$-form $$\star\vp_0=\d x^{4567}+\d x^{2367}+\d x^{2345}+\d x^{1357}-\d x^{1346}-\d x^{1256}-\d x^{1247}.$$ Here, we have started with the $3$-form and have shown how to define the other structures in terms of it; a fundamental fact of $G_2$-geometry however is that given \emph{any one} of $\vp_0$, $\star\vp_0$, $\t$ or $<,>$, we can always define the other structures. References for this information and an equivalent formulation of these structures arising from the octonions include \cite{BrGr, Br1, FeGr, Gray, HaLa, Jo, Kari}.

\begin{dfn}
A manifold $M$ is said to have a \emph{$G_2$-structure} if there is a $3$-form $\vp$ such that $(T_pM,\vp)\cong(\R^7,\vp_0)$ as vector spaces for every point $p\in M$. This is equivalent to a reduction of the tangent frame bundle from $GL(7,\R)$ to the Lie group $G_2$.
\end{dfn}

\begin{rem}
We will make frequent use of the fact that $\d^*\vp=0$ if and only if $\d\sv=0$. Also, because of the inclusion $G_2$ in $SO(7)$, all manifolds with $G_2$-structure are necessarily orientable; further, it can be shown that all manifolds with $G_2$-structure are spin, and any $7$-manifold with spin structure admits a $G_2$-structure.
\end{rem}

A natural geometric requirement is that $\vp$ be constant with respect to the Levi-Civita connection of the $G_2$-metric $g_{\vp}$ defined by $\vp$. In this case, the holonomy of $(M,\vp)$ is a subgroup of $G_2$, and $(M,\vp)$ is called a \emph{$G_2$-manifold}. The condition that $\na\vp=0$ is equivalent to $\d\vp=0$ and $\d^*\vp=0$ where $\d^*$ is the adjoint operator to the exterior derivative with respect to the Hodge star associated to the $G_2$-metric $g_{\vp}$. Fernandez and Gray \cite{FeGr} show that $G_2$-manifolds are just $1$ of $16$ types of $G_2$-structures on manifolds. Two of these classes include natural weakenings of the $G_2$-manifold requirements to manifolds with closed $G_2$-structures, $\d\vp=0$, and manifolds with coclosed $G_2$-structures, $\d^*\vp=0$.

In \cite{CST1}, we defined a \emph{$G_2$-morphism} to be a diffeomorphism $\Up:(M_1,\vp_1)\to(M_2,\vp_2)$ of manifolds with $G_2$-structures such that $\Up^*(\vp_2)=\vp_1$. Because $\d$ commutes with pullback maps, we get for free that $\d\vp_1=0$ if and only if $\d\vp_2=0$.

\begin{dfn}
Let $\Psi:(M_1,\vp_1)\to(M_2,\vp_2)$ be a diffeomorphism such that $\Psi^*(\star_2\vp_2)=\star_1\vp_1$. Then $\Psi$ is called a \emph{co$G_2$-morphism}. Again, since $\d$ commutes with pullback maps, we have $\d\stvo=0$ if and only if $\d\stvt=0$.
\end{dfn}

Let $(M_1,\vp_1)$ and $(M_2,\vp_2)$ be two $7$-dimensional manifolds with $G_2$-structures. Let $M_1\t M_2$ be the standard Cartesian product of $M_1$ and $M_2$ with canonical projection maps $\pi_i:M_1\t M_2\to M_i$. Define a $4$-form $\sv=\pi_1^*(\stvo)+\pi_2^*(\stvt)$. If both $\stvo$ and $\stvt$ are closed, then this form is also closed; in fact, for any $a_1,a_2\in\R$, $a_1\pi_1^*(\stvo)+a_2\pi_2^*(stvt)$ defines a (closed) $4$-form on $M_1\t M_2$. Taking $a_1=1$ and $a_2=-1$, we have the (closed) $4$-form $\widetilde{\sv}=\pi_1^*(\stvo)-\pi_2^*(\stvt)$.

\begin{thm}
A diffeomorphism $\Psi:(M_1,\vp_1)\to(M_2,\vp_2)$ is a co$G_2$-morphism if and only if $\widetilde{\sv}|_{\Ga_{\Psi}}\equiv 0$, where $\Ga_{\Psi}:=\{(p,\Psi(p))\in M_1\t M_2:p\in M_1\}$.
\end{thm}

\begin{proof}
The submanifold $\Ga_{\Psi}$ is the embedded image of $M_1$ in $M_1\t M_2$ with embedding given by $\tilde{\Psi}:M_1\to M_1\t M_2$ $\tilde{\Psi}(p)=(p,\Psi(p))$. Then $\widetilde{\sv}|_{\Ga_{\Psi}}=0$ if and only if $$0=\tilde{\Psi}^*(\widetilde{\sv})= \tilde{\Psi}^*\pi_1^*(stvo)-\tilde{\Psi}^*\pi_2^*(\stvt)$$ $$=(\pi_1\circ\tilde{\Psi})^*(\stvo)-(\pi_2\circ\tilde{\Psi})^*(\stvt)=(id_{M_1})^*(\stvo)-\Psi^*(\stvt)=\stvo-\Psi^*(\stvt).$$
\end{proof}

\section{Co$G_2$ Vector Fields, CoRochesterian $2$-Forms and CoRochesterian Vector Fields}
Let $M$ be a $7$-manifold with a $G_2$-structure. Recall that there is an action of the Lie group $G_2$ on the algebra of differential forms on $M$ from which we obtain decompositions of each space of $k$-forms on $M$ into irreducible $G_2$-representations. In particular, we can decompose the space of $3$-forms into the direct sum of a one-dimensional representation, a seven-dimensional representation and a $27$-dimensional representation, denoted from here on by $\Om^3_1$, $\Om^3_{7}$ and $\Om^3_{27}$ respectively; it is well known that $\Om^3_{7}=\{X\lrcorner\sv:X\in\X(M)\}$ where $\X(M)$ is the space of vector fields on $M$ (see for example \cite{FeGr, Jo, Kari, Sa}).

\begin{dfn}
Let $(M,\vp)$ be a manifold with coclosed $G_2$-structure.
\begin{enumerate}
    \item We define a \emph{coRochesterian $2$-form} $\si$ to be any $2$-form on $M$ such that $\d\si\in\Om^3_7(M)$. We denote the set of coRochesterian $2$-forms on $M$ by $\Om^2_{coRoc}(M)$.
    \item For a coRochesterian $2$-form $\si$, the vector field $X_{\si}$ satisfying $X_{\si}\lrcorner\sv=\d\si$ will be called a \emph{coRochesterian vector field}, and the set of coRochesterian vector fields on $M$ will be denoted by $\X_{coRoc}(M)$.
    \item A vector field $X$ is called a \emph{co$G_2$-vector field} if $\L_X(\sv)=0$. $\X_{coG_2}(M)$ will denote the set of all co$G_2$-vector fields.
\end{enumerate}
\end{dfn}

That $\X_{coG_2}$, $\Om^2_{coRoc}$ and $\X_{coRoc}$ are vector spaces follows immediately from the linearity properties of $\d$ and the interior product. We have the standard fact that $X$ is a co$G_2$-vector field if and only if $\d(X\lrcorner\sv)=0$ since $\d\sv=0$ implies that $\L_X(\sv)=\d(X\lrcorner\sv)+X\lrcorner\d\sv=\d(X\lrcorner\sv)$. Here, as in the case of closed $G_2$-structures, the map $\widetilde{\sv}:\X(M)\to\Om^3(M)$ given by $\widetilde{\sv}(X)=X\lrcorner\sv$ cannot be an isomorphism; however, by the nondegeneracy condition on the $4$-form $\sv$, we do have that $\widetilde{\sv}$ is injective, so for a given coRochesterian $2$-form $\si$, the associated coRochesterian vector field $X_{\si}$ is unique.

\begin{thm}
Every coRochesterian vector field on a manifold $M$ with coclosed $G_2$-structure $\vp$ is a co$G_2$-vector field. If every closed form in $\Om^3_7(M)$ is exact, then the spaces $\X_{coRoc}(M)$ and $\X_{coG_2}(M)$ coincide.
\end{thm}

\begin{proof}
The first statement follows immediately from the definitions. Next, for a co$G_2$-vector field $X$, $X\lrcorner\sv\in\Om^3_7(M)$ is closed, so, by assumption, there exists a $2$-form $\si$ with $X\lrcorner\sv=\d\si$.
\end{proof}

\begin{cor}
If $H^3(M)=\{0\}$, then every co$G_2$-vector field on a manifold with coclosed $G_2$-structure is a coRochesterian vector field.
\end{cor}

\begin{prop}
For any co$G_2$-vector fields $X_1$, $X_2$, there exists a $2$-form $\si$ such that $[X_1,X_2]\lrcorner\sv=\d\si$.
\end{prop}

\begin{proof}
\begin{equation*}
\begin{split}
[X_1,X_2]\lrcorner\sv&=\L_{X_1}(X_2\lrcorner\sv)-X_2\lrcorner(\underbrace{\L_{X_1}(\sv)}_{=0}) \\
&=\L_{X_1}(X_2\lrcorner\sv)=\d(X_1\lrcorner X_2\lrcorner\sv)+X_1\lrcorner(\underbrace{\d(X_2\lrcorner\sv)}_{=0})\\
&=\d(\sv(X_2,X_1,\cdot))
\end{split}
\end{equation*}
Thus, $[X_1,X_2]$ is a coRochesterian vector field with an associated $2$-form given by $\sv(X_2,X_1,\cdot,\cdot)$.
\end{proof}
Thus, we have the following inclusions of \emph{Lie algebras}: $$(\X_{coRoc}(M),[\cdot,\cdot])\subseteq(\X_{coG_2}(M),[\cdot,\cdot])\subseteq(\X(M),[\cdot,\cdot]).$$

For a coRochesterian $2$-form $\si$, the assignment $\si\mapsto X_{\si}$ where $X_{\si}$ is the unique associated coRochesterian vector field is linear. We now equip $\Om^2_{coRoc}(M)$ with a bracket as follows: for $\si,\tau\in\Om^2_{coRoc}(M)$, define $\{\si,\tau\}=\sv(X_{\si},X_{\tau},\cdot,\cdot)$. Then $\{\si,\tau\}\in \Om^2_{coRoc}(M)$ with coRochesterian vector field given by $[X_{\tau},X_{\si}]$ since $$\d(\{\si,\tau\})=\d(\sv(X_{\si},X_{\tau},\cdot,\cdot))=[X_{\tau},X_{\si}]\lrcorner\sv.$$

\begin{rem}
This bracket is again a specific case of the \emph{semibracket} defined in \cite{BHR} for the general multisymplectic setting, and a proof of the following result in this more general setting can be found in \cite[Proposition 3.7]{BHR}.
\end{rem}

\begin{prop}
For any $\si,\tau,\up\in\Om^2_{coRoc}(M)$,
$$\{\si,\{\tau,\up\}\}+\{\tau,\{\up,\si\}\}+\{\up,\{\si,\tau\}\}=\d(X_{\si}\lrcorner X_{\tau}\lrcorner\d\up)$$
\end{prop}

\begin{proof}
Let $\si,\tau,\up\in \Om^2_{coRoc}(M)$ with associated Rochesterian vector fields $X_{\si}$, $X_{\tau}$ and $X_{\up}$ respectively. Then we have the following:
\begin{equation*}
\begin{split}
\{\si&,\{\tau,\up\}\}+\{\tau,\{\up,\si\}\}+\{\up,\{\si,\tau\}\}=\{\si,\{\tau,\up\}\}-\{\tau,\{\si,\up\}\}-\{\{\si,\tau\},\up\}\\
&=X_{\si}\lrcorner X_{\{\tau,\up\}}\lrcorner\sv-X_{\tau}\lrcorner X_{\{\si,\up\}}\lrcorner\sv-X_{\{\si,\tau\}}\lrcorner X_{\up}\lrcorner\sv\\
&=X_{\si}\lrcorner\d\{\tau,\up\}-X_{\tau}\lrcorner\d\{\si,\up\}+[X_{\si},X_{\tau}]\lrcorner\d \up\\
&=X_{\si}\lrcorner\d(X_{\up}\lrcorner X_{\tau}\lrcorner\sv)-X_{\tau}\lrcorner\d(X_{\up}\lrcorner X_{\si}\lrcorner\sv)+[X_{\si},X_{\tau}]\lrcorner\d \up\\
&=-X_{\si}\lrcorner\d(X_{\tau}\lrcorner X_{\up}\lrcorner\sv)+X_{\tau}\lrcorner\d(X_{\si}\lrcorner X_{\up}\lrcorner\sv)+[X_{\si},X_{\tau}]\lrcorner\d \up\\
&=-X_{\si}\lrcorner\d(X_{\tau}\lrcorner\d \up)+X_{\tau}\lrcorner\d(X_{\si}\lrcorner\d \up)+[X_{\si},X_{\tau}]\lrcorner\d \up\\
&=-X_{\si}\lrcorner\d(X_{\tau}\lrcorner\d \up)+X_{\tau}\lrcorner\d(X_{\si}\lrcorner\d \up)+\L_{X_{\si}}(X_{\tau}\lrcorner\d \up)-X_{\tau}\lrcorner(\L_{X_{\si}}\d \up)\\
&=-X_{\si}\lrcorner\d(X_{\tau}\lrcorner\d \up)+X_{\tau}\lrcorner\d(X_{\si}\lrcorner\d \up)+X_{\si}\lrcorner\d(X_{\tau}\lrcorner\d \up)\\
&+\d(X_{\si}\lrcorner X_{\tau}\lrcorner\d \up)\underbrace{-X_{\tau}\lrcorner(X_{\si}\lrcorner\d\d \up)}_{=0}-X_{\tau}\lrcorner\d(X_{\si}\lrcorner\d \up)\\
&=X_{\si}\lrcorner\d(X_{\tau}\lrcorner\d \up)-X_{\si}\lrcorner\d(X_{\tau}\lrcorner\d \up)+X_{\tau}\lrcorner\d(X_{\si}\lrcorner\d \up)-X_{\tau}\lrcorner\d(X_{\si}\lrcorner\d \up)+\d(X_{\si}\lrcorner X_{\tau}\lrcorner\d \up)\\
&=\d(X_{\si}\lrcorner X_{\tau}\lrcorner\d \up)
\end{split}
\end{equation*}
\end{proof}

While we do not have a Lie algebra structure on $\Om^2_{coRoc}(M)$, we do, as noted above, still have a linear transformation $\Phi:\Om^2_{coRoc}(M)\to\X_{coRoc}(M)$. Assume that $\Phi(\si)=X_{\si}=0$, then $0=X_{\si}\lrcorner\sv=\d\si$ which implies that $\si$ is a closed $2$-form. Hence, coRochesterian vector fields are uniquely defined by their coRochesterian $2$-forms, up to the addition of a closed $2$-form.

\begin{thm}
\begin{enumerate}
    \item Given two coRochesterian $2$-forms $\si_1,\si_2\in\Om^2_{coRoc}(M)$, $\{\si_1,\si_2\}\in\ker\Phi$ if and only if $\d\si_1$ is constant along the flow lines of $X_{\si_2}$ if and only if $\d\si_2$ is constant along the flow lines of $X_{\si_1}$.
    \item Let $\psi:(M_1,\vp_1)\to(M_2,\vp_2)$ be a diffeomorphism. Then $\psi$ is a co$G_2$-morphism if and only if $\psi^*(\{\si,\tau\})=\{\psi^*\si,\psi^*\tau\}$ for all $\si,\tau\in\Om^2_{coRoc}(M_2)$.
\end{enumerate}
\end{thm}

\begin{proof}
\begin{enumerate}
    \item We show only the first equivalence since the second equivalence follows similarly. From the definition of the bracket, we have $$\{\si_1,\si_2\}=\sv(X_{\si_1},X_{\si_2},\cdot,\cdot)=X_{\si_2}\lrcorner(X_{\si_1}\lrcorner\sv)$$ $$=X_{\si_2}\lrcorner\d\si_1=\L_{X_{\si_2}}\si_1-\d(X_{\si_2}\lrcorner\si_1).$$ From this, we see that $\d\{\si_1,\si_2\}=\d\L_{X_{\si_2}}\si_1=\L_{X_{\si_2}}(\d\si_1)$. Then $\{\si_1,\si_2\}\in\ker\Phi$ if and only if $\L_{X_{\si_2}}(\d\si_1)=0$.
    \item First, assume that $\psi$ is a co$G_2$-morphism, and note that for $p\in M_1$, we have the maps
    \begin{equation*}
    \begin{split}
    &\d\psi_p:T_pM_1\to T_{\psi(p)}M_2\\
    &\psi^*_p:T^*_{\psi(p)}M_2\to T^*_pM_1\\
    &\d\psi^{-1}_{\psi(p)}=(\d\psi_p)^{-1}:T_{\psi(p)}M_2\to T_pM_1
    \end{split}
    \end{equation*}
    Since $\psi$ is a co$G_2$-morphism, $\psi^*(\stvt)=\stvo$ and $(\psi^{-1})^*(\stvo)=\stvt$, so by definition, for $p\in M_1$, we then get the following equivalent equations
    \begin{equation*}
    \begin{split}
    (\stvo)_p(\cdot,\cdot,\cdot,\cdot)&=\psi^*_p((\stvt)_{\psi(p)})(\cdot,\cdot,\cdot,\cdot) =(\stvt)_{\psi(p)}(\d\psi_p\cdot,\d\psi_p\cdot,\d\psi_p\cdot,\d\psi_p\cdot)\\
    (\stvt)_{\psi(p)}(\cdot,\cdot,\cdot,\cdot)&=(\psi^{-1}_{\psi(p)})^*((\stvo)_p)(\cdot,\cdot,\cdot,\cdot) =(\stvo)_p(\d\psi^{-1}_{\psi(p)}\cdot,\d\psi^{-1}_{\psi(p)}\cdot,\d\psi^{-1}_{\psi(p)}\cdot,\d\psi^{-1}_{\psi(p)}\cdot)
    \end{split}
    \end{equation*}
    Thus, we calculate for a coRochesterian $2$-form $\si\in\Om^2_{coRoc}(M_2)$ and vector fields $Y,Z,W$ on $M_1$,
    \begin{equation*}
    \begin{split}
    (X_{\psi^*\si}&\lrcorner\stvo)_p(Y_p,Z_p,W_p)=\d(\psi^*\si)_p(Y_p,Z_p,W_p)\\
    &=\psi^*_p(\d\si_{\psi(p)})(Y_p,Z_p,W_p)\\
    &=\psi^*_p((X_{\si}\lrcorner\stvt)_{\psi(p)})(Y_p,Z_p,W_p)\\
    &=\psi^*_p((\stvt)_{\psi(p)}((X_{\si})_{\psi(p)},\cdot,\cdot,\cdot))(Y_p,Z_p,W_p)\\
    &=(\stvt)_{\psi(p)}((X_{\si})_{\psi(p)},\d\psi_pY_p,\d\psi_pZ_p,\d\psi_pW_p)\\
    &=(\stvo)_p(\d\psi^{-1}_{\psi(p)}(X_{\si})_{\psi(p)},\d\psi^{-1}_{\psi(p)}(\d\psi_pY_p),\d\psi^{-1}_{\psi(p)}(\d\psi_pZ_p), \d\psi^{-1}_{\psi(p)}(\d\psi_pW_p))\\ 
    &=(\stvo)_p(\d\psi^{-1}_{\psi(p)}(X_{\si})_{\psi(p)},Y_p,Z_p,W_p)
    \end{split}
    \end{equation*}
    that is, $(X_{\psi^*\si})_p=\d\psi^{-1}_{\psi(p)}(X_{\si})_{\psi(p)}$. Hence we find that
    \begin{equation*}
    \begin{split}
    (\psi^*\{\si,\tau\})_p(Y_p,Z_p)&=(\psi^*(\stvt(X_{\si},X_{\tau},\cdot)))_p(Y_p,Z_p)\\
    &=\psi^*_p((\stvt)_{\psi(p)}((X_{\si})_{\psi(p)},(X_{\tau})_{\psi(p)},\cdot))(Y_p,Z_p)\\
    &=(\stvt)_{\psi(p)}((X_{\si})_{\psi(p)},(X_{\tau})_{\psi(p)},\d\psi_pY_p,\d\psi_pZ_p)\\
    &=(\stvo)_p(\d\psi^{-1}_{\psi(p)}(X_{\si})_{\psi(p)},\d\psi^{-1}_{\psi(p)}(X_{\tau})_{\psi(p)},\d\psi^{-1}_{\psi(p)}(\d\psi_pY_p), \d\psi^{-1}_{\psi(p)}(\d\psi_pZ_p))\\ 
    &=(\stvo)_p((X_{\psi^*\si})_p,(X_{\psi^*\tau})_p,Y_p,Z_p)\\
    &=(\stvo)_p((X_{\psi^*\si})_p,(X_{\psi^*\tau})_p,\cdot,\cdot)(Y_p,Z_p)=\{\psi^*\si,\psi^*\tau\}_p(Y_p,Z_p)
    \end{split}
    \end{equation*}
    Conversely, assume that $\psi^*(\{\si,\tau\})=\{\psi^*\si,\psi^*\tau\}$ for all $\si,\tau\in\Om^2_{coRoc}(M_2)$. Then, for any $\si,\tau\in\Om^2_{coRoc}(M_2)$ we have
    \begin{equation*}
    \begin{split}
    \psi^*(\{\si,\tau\})&=\psi^*(\stvt(X_{\si},X_{\tau},\cdot,\cdot))=\psi^*(X_{\tau}\lrcorner X_{\si}\lrcorner\stvt)=\psi^*(X_{\tau}\lrcorner\d\si)\\
    &=\psi^*(\d\si(X_{\tau},\cdot,\cdot))=\d\si(X_{\tau},\d\psi\cdot,\d\psi\cdot)=\d\si(\d\psi(\d\psi^{-1}X_{\tau}),\d\psi\cdot,\d\psi\cdot)\\
    &=(\psi^*\d\si)(\d\psi^{-1}X_{\tau},\cdot,\cdot)=(\d\psi^{-1}X_{\tau})\lrcorner(\psi^*\d\si)
    \end{split}
    \end{equation*}
    and
    \begin{equation*}
    \{\psi^*\si,\psi^*\tau\}=\stvo(X_{\psi^*\si},X_{\psi^*\tau},\cdot,\cdot)=X_{\psi^*\tau}\lrcorner\d(\psi^*\si)=X_{\psi^*\tau}\lrcorner(\psi^*\d\si)
    \end{equation*}
    which, by our hypothesis, yields that $\d\psi^{-1}X_{\tau}=X_{\psi^*\tau}$ for any $\tau\in\Om^2_{coRoc}(M_2)$. Then for any $\si\in\Om^2_{coRoc}(M_2)$, any vector fields $Y,Z,W\in\X(M_1)$ and $p\in M_1$,
    \begin{equation*}
    \begin{split}
    (X_{\psi^*\si}\lrcorner\stvo)_p(Y_p,Z_p,W_p)&=\d(\psi^*\si)_p(Y_p,Z_p,W_p)=(\psi^*\d\si)_p(Y_p,Z_p,W_p)\\
    &=\psi^*_p(X_{\si}\lrcorner\stvt)_p(Y_p,Z_p,W_p)\\
    &=(\stvt)_{\psi(p)}((X_{\si})_{\psi(p)},\d\psi_pY_p,\d\psi_pZ_p,\d\psi_pW_p)\\
    &=(\stvt)_{\psi(p)}(\d\psi_p(\d\psi^{-1}_{\psi(p)}(X_{\si})_{\psi(p)}),\d\psi_pY_p,\d\psi_pZ_p,\d\psi_pW_p)\\
    &=(\psi^*(\stvt))_{\psi(p)}(\d\psi^{-1}_{\psi(p)}(X_{\si})_{\psi(p)},Y_p,Z_p,W_p)\\
    &=(\psi^*(\stvt))_{\psi(p)}((X_{\psi^*\si})_p,Y_p,Z_p,W_p)
    \end{split}
    \end{equation*}
    \noindent
    Thus $X_{\psi^*\si}\lrcorner\stvo=X_{\psi^*\si}\lrcorner\psi^*(\stvt)$ which implies that $\stvo=\psi^*(\stvt)$ as desired.
\end{enumerate}
\end{proof}

\end{document}